\documentclass{amsart}
\usepackage{amsthm}

\global\overfullrule10pt

\theoremstyle{plain}
\newtheorem{theorem}{Theorem}
\newtheorem{corollary}[theorem]{Corollary}
\newtheorem{lemma}[theorem]{Lemma}
\newtheorem{proposition}[theorem]{Proposition}
\theoremstyle{definition}
\theoremstyle{remark}
\newtheorem{remark}[theorem]{Remark}
\newtheorem*{remark*}{Remark}

\newcommand\CC{{\mathbf C}}
\newcommand\RR{{\mathbf R}}
\newcommand\DD{{\mathbf D}}
\newcommand\BB{{\mathbf B}}
\newcommand\spr[1]{\langle#1\rangle}
\newcommand\hol{_{\text{hol}}}
\renewcommand\Re{\operatorname{Re}}

\newcommand\oz{\overline z}
\newcommand\pOm{{\partial\Omega}}
\newcommand\oOm{{\overline{\Omega}}}
\newcommand\dbar{\overline\partial}
\newcommand\ok{{\overline k}}
\newcommand\LL{\mathcal L}
\newcommand\HH{{\mathbf H}}
\newcommand\MM{{\mathbf M}} %[
\newcommand\nuj{(0,1]} %)

\newcommand\Ric{\operatorname{Ric}}
\newcommand\dist{\operatorname{dist}}

\newcommand\pam{{\partial\MM}}
\newcommand\od{\overline d}
\newcommand\oZ{\overline Z}
\newcommand\rc{{\rho(c)}}
\newcommand\ox{\overline x}

\allowdisplaybreaks[4]

\begin{document}

\title[Balanced metrics]{Radial balanced metrics on the unit ball
 of the Kepler manifold}
\author[H.~Bommier-Hato]{H\'el\`ene Bommier-Hato}
\address{Faculty of Mathematics, University of Vienna, Oskar-Morgenstern-Platz~1,
 1090~Vienna, Austria}
% Aix-Marseille Universit\'e, I2M UMR CNRS~7373,
% 39~Rue F.~Joliot-Curie, 13453~Marseille Cedex~13, France
\email{helene.bommier@gmail.com}
\thanks{H.~Bommier was supported by the FWF project P~30251-N35.}
\author[M.~Engli\v s]{Miroslav Engli\v s}
\address{Mathematics Institute, Silesian University in Opava,
 Na~Rybn\'\i\v cku~1, 74601~Opava, Czech Republic {\rm and }
 Mathematics Institute, \v Zitn\' a 25, 11567~Prague~1,
 Czech Republic}
\email{englis{@}math.cas.cz}
\thanks{Research supported by GA\v CR grant no.~16-25995S
 and RVO funding for I\v CO~67985840.}
\author[E.-H.~Youssfi]{El-Hassan Youssfi}
\address{Aix-Marseille Universit\'e, I2M UMR CNRS~7373,
 39 Rue F-Juliot-Curie, 13453~Marseille Cedex 13, France}
\email{el-hassan.youssfi@univ-amu.fr}
% 30-XX Functions of a complex variable {For analysis on manifolds, see 58-XX}
%  30Fxx Riemann surfaces
%   30F45 Conformal metrics (hyperbolic, Poincare, distance functions)
% 32-XX	Several complex variables and analytic spaces
%  32Axx Holomorphic functions of several complex variables
%   32A25 Integral representations; canonical kernels (Szego, Bergman, etc.)
%   32A36 Bergman spaces
%  32Fxx Geometric convexity
%   32F45 Invariant metrics and pseudodistances
%  32Qxx Complex manifolds
%   32Q15 Kahler manifolds
%  32Txx Pseudoconvex domains
%   32T15 Strongly pseudoconvex domains
% 46-XX Functional analysis
%  46Exx Linear function spaces and their duals
%   46E22 Hilbert spaces with reproducing kernels
% 53-XX Differential geometry
%  53Bxx Local differential geometry
%   53B35 Hermitian and Kahlerian structures
%  53Cxx Global differential geometry
%   53C55 Hermitian and Kahlerian manifolds
\begin{abstract}
We~show that there is no radial balanced metric on the unit ball of the
Kepler manifold with not too wild boundary behavior. Additionally,
we~identify explicitly the weights corresponding to radial metrics
with such boundary behavior which satisfy the balanced condition as
far as germs at the boundary are concerned.
Related results for Poincar\'e metrics are also established.
\end{abstract}

\maketitle

\section{Introduction}\label{sec1}
Let $\Omega$ be a domain in $\CC^n$, $n\ge1$, or on an $n$-dimensional
complex manifold, and $\Phi$ a strictly plurisubharmonic function on $\Omega$
with the associated K\"ahler form $\omega=\frac i2\partial\dbar\Phi$ and
volume element $\wedge^n\omega$. The~weighted Bergman space $L^2\hol(\Omega,
e^{-(n+1)\Phi}\hbox{$\wedge^n\omega$})$ of all holomorphic functions on $\Omega$
square integrable with respect to $e^{-(n+1)\Phi}\wedge^n\omega$ is well known
to have bounded point evaluations and hence possesses a reproducing kernel
$K_{e^{-(n+1)\Phi}\wedge^n\omega}(x,y)$ (a~weighted Bergman kernel).
The~K\"ahler metric associated to~$\omega$ --- or,~abusing terminology,
the function~$u:=e^{-\Phi}$ ---  is~called \emph{balanced}~if
\begin{equation}
 K_{e^{-(n+1)\Phi}\wedge^n\omega}(z,z) = c e^{(n+1)\Phi(z)}
 \qquad\forall z\in\Omega,  \label{tTA}
\end{equation}
that~is,
\begin{equation}
 K_{u^{n+1}\wedge^n(\frac i2\partial\dbar\log\frac1u)}(z,z)
 = \frac c{u(z)^{n+1}} \qquad\forall z\in\Omega   \label{tTB}
\end{equation}
for some constant~$c$. (One~can check that this condition indeed depends
only on the K\"ahler form~$\omega$, not on its potential $\Phi$ or,
equivalently, on $u=e^{-\phi}$.)

The~notion extends in an obvious way also to the more general setting of
functions replaced by sections of line bundles: namely, if~$\LL$ is a
holomorphic Hermitian line bundle over $\Omega$ with K\"ahler connection
$\nabla$ such that $\operatorname{curv}\nabla=\omega$, let
$L^2\hol(\LL^*,\wedge^n\omega)$ be the Bergman space of all square-integrable
holomorphic sections of its dual budle~$\LL^*$, and for~any orthonormal
basis $\{s_j\}$ of this space, set
\begin{equation}  \epsilon(x) := \sum_j \|s_j(x)\|_x^2 .   \label{tTC}
\end{equation}
One~can again show that $\epsilon(x)$ does not depend on the choice of
the orthonormal basis $\{s_j\}$ and also does not depend on the line
bundle $\LL$ but only on the K\"ahler form~$\omega$. The~K\"ahler form,
or~the associated K\"ahler metric, is~called \emph{balanced}~if
\begin{equation}  \epsilon \equiv \text{const.}    \label{tTD}
\end{equation}
When $\LL$ is trivial, its~sections can be identified with functions
on~$\Omega$, and one recovers the situation from the previous paragraph.

The function $\epsilon$ and the condition \eqref{tTD} have appeared in
the literature under different names, cf.~Rawnsley \cite{Raw}, Cahen,
Gutt and Rawnsley~\cite{CGR}, Kempf~\cite{Kempf} and Ji~\cite{Ji},
or~Zhang~\cite{ZhaS}; the term \emph{balanced} was first used by
Donaldson~\cite{Donald}, who also established the existence of such
metrics on any (compact) projective K\"ahler manifold with constant
scalar curvature. Subsequent studies of the existence and uniqueness of
balanced metrics in the compact case include Seyyedali~\cite{Seyd},
Li~\cite{ChiLi}, and others; see~also Phong and Sturm \cite{PhoStu}
for an overview.

However, despite the extensive studies of the compact case, much less
seems to be known concerning existence and uniqueness of balanced metrics
in the noncompact setting of domains in~$\CC^n$ or on complex manifolds.
Beside the simplest example, which is the Bergman metric on the unit ball
$\BB^n$ of~$\CC^n$, corresponding~to
\begin{equation} u(z)=(1-|z|^2)^\alpha, \quad\alpha>\tfrac n{n+1},
 \label{tTE} \end{equation}
the only known examples of balanced metrics are the
(appropriate multiples of~the) Bergman metrics on bounded symmetric
domains in~$\CC^n$, or, more generally, of~invariant metrics on bounded
homogeneous domains; and the flat (Euclidean) metric on~$\CC^n$
(with $\Phi(z)=|z|^2$). Miscellaneous partial results concerning uniqueness
and/or existence of balanced metrics on domains are due to Loi~and
Zedda~\cite{LoiZe}, Cuccu and Loi \cite{LoiX}, Greco and Loi \cite{LoiY},
Arezzo and Loi~\cite{ArLoi}, or the present authors \cite{BEY}.
It~is a conjecture of the author's \cite{EKokyu} that for
$\Omega\subset\CC^n$ bounded strictly pseudoconvex with smooth boundary
and any $\alpha>n$, there exists a unique balanced metric
on~$\Omega$ with $u^{(n+1)/\alpha}$ vanishing precisely to the first order
at the boundary~$\pOm$, i.e.~$u(z)\asymp\dist(z,\pOm)^{\alpha/(n+1)}$.
For $\alpha=n+1$ and $u$ radial on the unit disc $\DD=\{z\in\CC:|z|<1\}$, %(
i.e.~$u(z)=f(|z|^2)$ for some $f\in C^\infty[0,1)$, %]
this problem was considered in~\cite{EbD}, where it was shown that among
all $f$ with sufficiently nice boundary behavior, the~only one giving rise
to a balanced metric is~\eqref{tTE}, i.e.~$f(t)=1-t$.

In~this paper, we~consider the above problem, again with $\alpha=n+1$,
in the setting of the unit disc replaced by the unit ball
$\MM=\MM^n:=\{z\in\HH:|z|<1\}$ of the Kepler manifold
\begin{equation}
 \HH=\HH^n:=\{z\in\CC^{n+1}:z\cdot z=0,\; z\neq0\},
  \qquad n\ge2  \label{tTF}
\end{equation}
(here and throughout, $z\cdot w:=\sum_j z_j w_j$).
The~latter is an $n$-dimensional complex submanifold in~$\CC^{n+1}$,
which can be identified as a symplectic manifold with the cotangent bundle
(minus its zero section) of~the unit sphere $\mathbf S^n\subset\RR^{n+1}$.
The~origin is a removable singularity for~$\HH$, i.e.~$\HH\cup\{0\}$
is a normal complex analytic space, and in fact is the simplest example
of Jordan-Kepler varieties \cite{EUjk} which generalize the classical
determinantal varieties. If~the conjecture in the previous paragraph
is~valid, any~balanced metric on $\MM$ has to be rotation-invariant,
hence we will again be looking for $u=e^{-\Phi}$ in the form %(
$u(z)=f(|z|^2)$ for some $f\in C^\infty[0,1)$ which vanishes precisely %]
to the first order at~$1$, i.e.~$f(1^-)=0$ and $f'(1^-)\neq0$; replacing
$f$ by a suitable multiple thereof, we~may assume that $f'(1^-)=-1$.
As~in \cite{EbD} for the disc, our~strategy will be to look at the
boundary behavior of both sides of~\eqref{tTB}. Our~main results are
the following; for~simplicity they are formulated for the simplest
case $n=2$, but~the methods carry over to general~$n$.

\begin{theorem}\label{pA}
Assume that $f\in C^\infty\nuj$ satisfies $f(1)=0$, $f'(1)=-1$ and there is
a constant $c\neq0$ such that
\begin{equation}
 \frac c{u(z)^3} - K_{u^3\wedge^2(\frac i2\partial\dbar\log\frac1u)}(z,z),
 \qquad u(z):=f(|z|^2),  \label{tTG}
\end{equation}
is smooth on~$\MM^2$ up~to $|z|=1$. Then
\begin{equation}
 t f' (f f'+t f f''-t f^{\prime2}) = \phi_v + h ,  \label{tTH}
\end{equation}
where $h\in C^\infty\nuj$ satisfies $h^{(k)}(1)=0$ $\forall k$, and
\begin{align}
\phi_v(t) &= e^{L/4}\Big(\cosh\frac{L\sqrt v}4
  - \frac1{\sqrt v}\sinh\frac{L\sqrt v}4\Big),
  \qquad L:=\log\frac1t, \nonumber \\
&= \frac{(1+\sqrt v)t^{(-1+\sqrt v)/4}-(1-\sqrt v)t^{(-1-\sqrt v)/4}}
 {2\sqrt v}  \label{tTZ}
\end{align}
for some $v\in\RR$.
\end{theorem}

Here for $v=0$, \eqref{tTZ} is to be interpreted as the limit $v\to0$,
i.e.~as $e^{L/4}(1-\frac L4)=t^{-1/4}(1+\frac14\log t)$. Note also that
the right-hand side of \eqref{tTZ} remains unchanged upon replacing
$\sqrt v$ by $-\sqrt v$, so~there is no ambiguity connected with the
choice of the square root~$\sqrt v$.
% (Notice~also that $\phi_v\equiv1$ for $v=1$.)

Finding balanced metrics on $\MM$ thus reduces to looking for solutions
of the differential equation~\eqref{tTH}.

\begin{corollary}\label{pB} If $u(z)=f(|z|^2)$ on $\MM^2$ with $f(1)=0$,
$f'(1)=-1$ and $f$ real-analytic at~1, then the balanced condition
\eqref{tTB} is never satisfied. That~is, there exists no balanced $u$
which would be real-analytic up to the boundary $|z|=1$. \end{corollary}

The~last corollary is in contrast with the situation for the disc,
where $f(t)=1-t$ gives a balanced metric. We~also remark that for the
disc in~\cite{EbD}, we~were able to identify explicitly the functions
satisfying~\eqref{tTG}; this we are unable to offer here, though some
observations are presented in Section~\ref{sec52} below.
This also explains why a different method is needed here than in~\cite{EbD}:
we~first identify the weights $u^3\wedge^2(\frac i2\partial\dbar\log\frac1u)$,
which actually turn out to be given by the formula~\eqref{tTH},
and then are able to draw conclusions even without explicit knowledge
of the solutions.
% For~this reason, furthermore, a~different method of proof is needed
% here than in~\cite{EbD}.

An~analogue of Theorem~\ref{pA} can be proved also for the disc, yielding
another (though not much simpler) proof of Theorem~1 in~\cite{EbD}.
(We~pause to note that in~\cite{EbD}, the~condition that $u(z)^{n+1}
K_{u^{n+1}\wedge^n(\frac i2\partial\dbar\log\frac1u)}(z,z)-c$ be smooth
up to the boundary and vanish to second order there was used instead of
the smoothness of \eqref{tTG} up to the boundary; our~Theorem~\ref{pA}
remains in force also for this modification, with the same proof.)

Finally, similarly as for the disc, the hypothesis of the smoothness
of~$f$ at~$t=1$ in Theorem~\ref{pA} can be weakened considerably:
writing temporarily for brevity $r(z):=\dist(z,\pOm)$, assume that
$u\in C^\infty(\Omega)$ has an asymptotic expansion at $\pOm$ of the form
\begin{equation}  u(z) \approx r(z) \sum_{k=0}^\infty \sum_{j=0}^{M_k}
 a_{kj}(z) r(z)^k (\log r(z))^j,   \label{tTI}   \end{equation}
with some nonnegative integers $M_k$ and functions
$a_{kj}\in C^\infty(\oOm)$, where
\begin{equation}
  M_0=0 \qquad\text{and}\qquad a_{00}=1\text{ on }\pOm.  \label{tTJ}
\end{equation}
Here (\ref{tTI}) means that $u$ differs from the partial sum
$\sum_{k=0}^{N-1}$ of the right-hand side by a function in $C^N(\oOm)$
all of whose partial derivatives up to order $N$ vanish at~$\pOm$,
for all $N=0,1,2,\dots$.
Note that $u\in C^\infty(\oOm)$ is equivalent to $M_k=0$~$\forall k$.

\begin{theorem}\label{pC} Assume that $u(z)=f(|z|^2)$ is a smooth radial
function on~$\MM^2$, with asymptotic expansion \eqref{tTI}
satisfying~\eqref{tTJ}, for which \eqref{tTG} is smooth up to~$|z|=1$.
Then $f\in C^\infty\nuj$ $($and,~hence, Theorem~\ref{pA} applies$)$.
\end{theorem}

In~conclusion, if~there exists a radial balanced metric on~$\MM^2$,
then either it has more complicated behavior at the boundary than
given by~\eqref{tTI} and~\eqref{tTJ}; or~it involves a nonzero ``flat''
piece corresponding to the function~$h$ in~\eqref{tTH}, so that
$f\in C^\infty\setminus C^\omega[0,1]$. Observe that any nonzero $h$
in \eqref{tTH} must have some kind of singularity at $t=1$; an~example
of such function is $h(t)=e^{-1/\sqrt{1-t}}$.

We~pause to remark that weight functions with the boundary singularity
given by \eqref{tTI} and~\eqref{tTJ} occur naturally in the analysis on
strictly pseudoconvex domains: allowing a power at $r$ in front of the
double sum in~\eqref{tTI} (or~taking the appropriate root of the
weight function), examples of functions of the form \eqref{tTI}
and~\eqref{tTJ} include the Bergman kernel on the diagonal,
the Szeg\"o kernel on the diagonal, the potential of the Poincar\'e
metric (i.e.~the solution of the Monge-Amp\'ere equation,
cf.~Section~\ref{sec52} below), the~Bergman invariant,
and so forth~\cite{Emona}. The~necessity for a balanced metric to have
more complicated boundary behavior would therefore seem a bit surprising.

The~proofs of Theorem~\ref{pA} and Corollary~\ref{pB} are given in
Section~\ref{sec3}, after recalling some preliminaries about the
Kepler manifold in Section~\ref{sec2}. The~proof of Theorem~\ref{pC}
occupies Section~\ref{sec4}. The~final section, Section~\ref{sec5},
collects some concluding comments and remarks concerning completeness
of balanced metrics, an~identification of a family of Poincar\'e
(i.e.~K\"ahler-Einstein) metrics on~$\MM$, as~well as a correction of
a small overlook in the proof of a theorem in~\cite{EbD}.

\begin{remark*} Our~terminology in this paper is perhaps a bit at odds
with common usage due to the powers $n+1$ in \eqref{tTA} and~\eqref{tTB}:
namely, in~most of the literature one calls a K\"ahler form $\omega$
balanced if the reproducing kernel with respect to the volume element
$e^{-\Phi}\wedge^n\omega$ equals $c e^\Phi$ on the diagonal.
Thus $\omega$ is balanced in the sense of \eqref{tTA} if and only if
$(n+1)\omega$ is balanced in the sense of the preceding sentence.
Of~course, our~reason for this small deviation was to avoid having to
keep track of the factor $(n+1)$ all the~time.   \qed \end{remark*}

\section{Preliminaries}\label{sec2}
Recall that the Kepler manifold $\HH$ is the orbit of the vector
$e=(1,i,0,\dots,0)$ under the $O(n+1,\CC)$-action on~$\CC^{n+1}$;
its~unit ball $\MM$ as well as the outer boundary $\pam=\{z\in\HH:
|z|=1\}$ of the latter are invariant under $O(n+1,\CC)\cap U(n+1)=O(n+1,\RR)$,
and in fact $\pam$ is the orbit of $e$ under~$O(n+1,\RR)$.
In~particular, there is a unique $O(n+1,\RR)$-invariant probability
measure $d\mu$ on~$\pam$, coming from the Haar measure on the (compact)
group $O(n+1,\RR)$. Explicitly, denoting
\begin{equation}
 \alpha := (n+1) \frac{(-1)^{j-1}}{z_j}\,dz_1\wedge\dots\wedge\widehat{dz_j}
 \wedge\dots\wedge dz_{n+1} \qquad\text{on } z_j\neq0  \label{tTK}
\end{equation}
(this~is, up~to constant factor, the unique $SO(n+1,\CC)$-invariant
holomorphic $n$-form on~$\HH$, see~\cite{OPY}) and defining a $(2n-1)$-form
$\eta$ on $\pam$~by
$$ \eta(z)(V_1,\dots,V_{2n-1}) := \alpha(z)\wedge\overline{\alpha(z)}
 (z,V_1,\dots,V_{2n-1}), \qquad V_1,\dots,V_{2n-1}\in T_z(\pam),  $$
we~then have $d\mu=|\eta|/|\eta|(\pam)$, where $|\eta|$ denotes the
measure induced by $\eta$ on~$\pam$. Note that $\MM$, $\pam$ and $d\mu$
are also invariant under the complex rotations
$$ z\mapsto e^{i\theta} z, \qquad \theta\in\RR.   $$

\begin{proposition}\label{pD} For a function $u(z)=f(|z|^2)$ on~$\MM$,
$$ u^{n+1}\bigwedge\nolimits^n\Big(\frac i2\partial\dbar\log\frac1u\Big)
 = \frac{W[f]}{(n+1)^2} \, \frac{\alpha\wedge\overline\alpha}
 {(-1)^{n(n+1)/2}(2i)^n} ,  $$
where $\alpha$ is given by \eqref{tTK} and
\begin{equation}
 W[f] := (-1)^n t f^{\prime(n-1)} (ff'+tff''-tf^{\prime2}) . \label{tTL}
\end{equation}
$($Here and throughout, $f,f',f''$ and $W[f]$ are evaluated at $t=|z|^2.)$
\end{proposition}

\begin{proof} Working e.g.~in the local chart $z_{n+1}\neq0$, we~choose
the local coordinates $z=(Z,z_{n+1})$ on~$\HH$, with $Z:=(z_1,\dots,z_n)$
and $z_{n+1}=\pm\sqrt{-Z\cdot Z}$. Then $u(z)=f(|z|^2)=f(|Z|^2+|Z\cdot Z|)$.
By~elementary linear algebra,
$$ u^{n+1}\bigwedge\nolimits^n\Big(\frac i2\partial\dbar\log\frac1u\Big)
 = J[u](\tfrac i2)^n dz_1\wedge d\oz_1\wedge\dots\wedge dz_n\wedge d\oz_n, $$
where $J[u]$ is the Monge-Amp\'ere determinant
$$ J[u]=(-1)^n \det\begin{bmatrix} u & {\partial u}/{\partial z_k} \\
    {\partial u}/{\partial\oz_j}
     & {\partial^2 u}/{\partial z_k\partial\oz_j} \end{bmatrix}
   _{j,k=1}^n .   $$
In~our case,
$$ \frac{\partial u}{\partial\oz_j} = f'(|Z|^2+|Z\cdot Z|) d_j,
 \quad d_j:=z_j + \frac{Z\cdot Z}{|Z\cdot Z|}\oz_j,  $$
and similarly $\frac{\partial u}{\partial z_k}=f'(|z|^2)\od_k$; while
$$ \frac{\partial^2 u}{\partial z_k\partial\oz_j} = d_j\od_k f''(|z|^2)
 + \Big(\delta_{jk}+\frac{\oz_j z_k}{|Z\cdot Z|}\Big) f'(|z|^2).  $$
Omitting the argument $|z|^2$ at $f$ and its derivatives, we~thus~get
\begin{align*}
(-1)^n J[u] &= \det\begin{bmatrix} f & \od_k f' \\ d_j f' &
  d_j\od_k f''+ (\delta_{jk}+\frac{\oz_j z_k}{|Z\cdot Z|})f'\end{bmatrix} \\
&= \det\begin{bmatrix} f & 0 \\ d_j f' &
  d_j\od_k f''+ (\delta_{jk}+\frac{\oz_j z_k}{|Z\cdot Z|})f'
  - f^{\prime2}\frac{d_j\od_k}f \end{bmatrix} \\
&= f \det\begin{bmatrix} (f''-\frac{f^{\prime2}}f)d_j\od_k
  + (\delta_{jk}+\frac{\oz_j z_k}{|Z\cdot Z|})f' \end{bmatrix}_{j,k=1}^n \\
&= f f^{\prime n} \det\begin{bmatrix} \left(\frac{f''}{f'}-\frac{f'}f\right)d_j\od_k
  + (\delta_{jk}+\frac{\oz_j z_k}{|Z\cdot Z|}) \end{bmatrix} \\
&= f f^{\prime n} \det\begin{bmatrix} I+\dfrac{\spr{\cdot,Z}Z}{|Z\cdot Z|}
  + \left(\frac{f''}{f'}-\frac{f'}f\right) \spr{\cdot,d}d \end{bmatrix} ,
\end{align*}
where $d$ is the vector $(d_1,\dots,d_n)=Z+\frac{Z\cdot Z}{|Z\cdot Z|}\oZ$.
Passing to a basis containing $Z,d$ shows that the last determinant equals
$$ \Big(1+\frac{|Z|^2}{|Z\cdot Z|}\Big)
   \Big(1+\left(\frac{f''}{f'}-\frac{f'}f\right)|d|^2\Big)
   - \left(\frac{f''}{f'}-\frac{f'}f\right)
    \Big|\Big\langle \frac Z{|Z\cdot Z|^{1/2}},d\Big\rangle\Big|^2   $$
(which formula remains in force even if $d,Z$ are linearly dependent).
Since $|d|^2=2|Z|^2+2|Z\cdot Z|=2|z|^2$ and
$$ \Big\langle \frac Z{|Z\cdot Z|^{1/2}},d\Big\rangle =
 \frac{|Z|^2+|Z\cdot Z|}{|Z\cdot Z|^{1/2}}
 = \frac{|z|^2}{|Z\cdot Z|^{1/2}}   $$
while $|Z\cdot Z|=|z_{n+1}|^2$, the determinant equals
$$ \frac{|z|^2}{|Z\cdot Z|} \Big(1+2|z|^2 \left(\frac{f''}{f'}-\frac{f'}f\right)\Big)
 - \frac{|z|^4}{|Z\cdot Z|} \left(\frac{f''}{f'}-\frac{f'}f\right)
 = \frac{|z|^2}{|z_{n+1}|^2} \Big(1+|z|^2 \left(\frac{f''}{f'}-\frac{f'}f\right)\Big). $$
Now by~\eqref{tTK}
$$ \frac{\alpha\wedge\overline\alpha} {(-1)^{n(n+1)/2}(2i)^n}
 = \frac{(n+1)^2}{|z_{n+1}|^2} (\tfrac i2)^n
   dz_1\wedge d\oz_1\wedge\dots\wedge dz_n\wedge d\oz_n ,  $$
so,~switching to the notation $|z|^2=:t$,
$$ J[u](\tfrac i2)^n dz_1\wedge d\oz_1\wedge\dots\wedge dz_n\wedge d\oz_n
 = \frac{(-1)^n t f f^{\prime n}(1+t\left(\frac{f''}{f'}-\frac{f'}f\right))}{(n+1)^2}
  \frac{\alpha\wedge\overline\alpha} {(-1)^{n(n+1)/2}(2i)^n} ,  $$
proving the claim.  \end{proof}

The~measure
$$ d\rho(z) := \frac{\alpha\wedge\overline\alpha} {(-1)^{n(n+1)/2}(2i)^n}  $$
admits a handy description in ``polar coordinates'': namely, it~was shown in
\cite[Lemma~2.1]{MeY} that for any measurable function $f$ on~$\HH$,
$$ \int_\HH f(z) \,d\rho(z) = c_\MM \int_0^\infty \int_\pam f(\sqrt t \zeta)
 \,t^{n-2} \,d\mu(\zeta) \, dt ,  $$
where
$$ c_\MM = (n-1) \int_\MM d\rho. $$
As~in Theorem~5 on page 273 in~\cite{BEY}, it~then follows, in~particular,
that for any nonnegative integrable function $\phi$ on ~$(0,1)$,
the~reproducing kernel $K_{\phi\,d\rho}(x,y)$ of the weighted Bergman space
$L^2\hol(\MM,\phi\,d\rho)$ on $\MM$ is given~by
$$ K_{\phi\,d\rho}(x,y) = \frac1{c_\MM} \sum_{k=0}^\infty
 \frac{N(k)}{c_{k+n-2}} (x\cdot\overline y)^k ,  $$
where
$$ N(k) := \binom{k+n-1}{n-1}+\binom{k+n-2}{n-1}  $$
while
$$ c_k := \int_0^1 t^k \phi(t) \,dt  $$
are the moments of the function~$\phi$.

\begin{corollary} \label{pE} For $u(z)=f(|z|^2)$ on~$\MM$,
$$ K_{u^{n+1}\wedge^n(\frac i2\partial\dbar\log\frac1u)}(z,z)
 = \frac{(n+1)^2}{c_\MM} F(|z|^2),  $$
where
$$ \postdisplaypenalty1000000
F(t) := \sum_{k=0}^\infty \frac{N(k)}{c_{k+n-2}} t^k , $$
with $N(k)$ and $c_k$ as above, for $\phi=W[f]$ given by~\eqref{tTL}.
\end{corollary}

\section{The smooth case}\label{sec3}
Specializing the last corollary to $n=2$, we~see in particular that
$$ K_{u^3\wedge^2(\frac i2\partial\dbar\log\frac1u)}(z,z)
=\frac9{c_\MM} F(|z|^2),  $$
where
\begin{equation}
 F(t) = \sum_{k=0}^\infty \frac{2k+1}{\int_0^1 t^k\phi(t)\,dt} t^k
 \label{tUA}
\end{equation}
with $\phi=W[f]=tf'(ff'+tff''-tf^{\prime2})$. The~balanced condition
therefore reads
\begin{equation} F(t) = \frac c{f(t)^3} \qquad\forall t\in(0,1)
 \label{tUB}
\end{equation}
for some nonzero constant~$c$ (differing from the one in \eqref{tTB}
by the factor $9/c_\MM$), while the hypothesis of Theorem~\ref{pA}
is simply that
\begin{equation} F-\frac c{f^3}\in C^\infty\nuj \label{tUC}
\end{equation}
with some nonzero constant~$c$ (the~same one as in the preceding formula).

\begin{proof}[Proof of Theorem~\ref{pA}]
The hypothesis on $f$ implies that $\phi\in C^\infty\nuj$ and
$\phi(1)=1$. In~terms of the variable $L:=\log\frac1t$, we~thus have
$$ \phi(t) \approx \sum_{j=0}^\infty a_j L^j, \qquad a_0=1,   $$
where $a_j:=\frac1{j!}\frac{d^j}{dL^j} \phi(e^{-L})|_{L=0}$ are some real numbers,
and ``$\approx$'' means that $\phi-\sum_{j=0}^{N-1}$ vanishes to order at least
$N$ at $t=1$, for each $N=0,1,2,\dots$. Using the formula
$$ \int_0^1 t^k L^j \, dt = \frac{j!}{(k+1)^{j+1}},   $$
this implies that
$$ c_k = \int_0^1 t^k \phi(t) \,dt \approx \sum_{j=0}^\infty \frac{j!a_j}{(k+1)^{j+1}}, $$
where ``$\approx$'' now means that $c_k-\sum_{j=0}^{N-1}$ is $O(k^{-N-1})$ as $k\to+\infty$,
for each $N=0,1,2,\dots$. Taking reciprocal gives
\begin{equation}
 \frac1{c_k} \approx (k+1) \sum_{m=0}^\infty \frac{A_m}{(k+1)^m} , \quad A_0=1,
 \label{iv5}
\end{equation}
where for $m\ge1$
\begin{align}
A_m &= \sum_{n=1}^m (-1)^n \sum_{\begin{subarray}{c} j_1,\dots,j_n\ge1\\j_1+\dots+j_n=m\end{subarray}}
 \prod_{j=1}^n (j_i!a_{j_i})  \nonumber \\
&= -m!a_m + (\text{a polynomial in }a_1,\dots,a_{m-1}) .  \label{iv6}
\end{align}
Thus
\begin{align}
F(t) &= \sum_{k=0}^\infty \frac{2k+1}{c_k}t^k
 \approx \sum_{k=0}^\infty (2(k+1)^2-(k+1)) \sum_{m=0}^\infty \frac{A_m}{(k+1)^m}  \nonumber \\
&= \sum_{m=0}^\infty A_m (2\Phi(t,m-2,1)-\Phi(t,m-1,1))  \nonumber \\
&= \sum_{m=-2}^\infty (2A_{m+2}-A_{m+1}) \Phi(t,m,1) \qquad (A_{-1}:=0)  \nonumber \\
&= \frac{2(1+t)}{(1-t)^3} + \frac{2A_1-1}{(1-t)^2} + \frac{2A_2-A_1}{1-t} \nonumber \\
&\hskip4em  + \sum_{m=1}^\infty \frac{2A_{m+2}-A_{m+1}}t
    \Big[\frac{(-1)^m}{(m-1)!} L^{m-1} \log L + h_m(L) \Big] ,  \label{tXA}
\end{align}
where $h_m$ is holomorphic in the disc~$|L|<2\pi$; here $\Phi(t,s,v)$ stands for the
Lerch transcendental function~\cite[\S1.11]{BErd}
$$ \Phi(z,s,v) := \sum_{k=0}^\infty \frac{z^k}{(k+v)^s},
 \qquad s\in\CC, \quad v\neq0,-1,-2,\dots,   $$
 and we have made use of Lerch's formula \cite[1.11(9)]{BErd}
\begin{equation}  t\Phi(t,m,1) = \frac{(-1)^m}{(m-1)!} L^{m-1} \log L
 + \sum_{k=0}^\infty{}' \frac{(-1)^k}{k!}\zeta(m-k) L^k ,
 \qquad L:=\log\frac1t ,   \label{tUL}  \end{equation}
valid for an integer $m\ge1$, where the sum on the right-hand side converges for $|L|<2\pi$,
and the $\sum'$ means that in the term $k=m-1$, $\zeta(1)$ should
be replaced by $\sum_{j=1}^{m-1}\tfrac1j$.

By~hypothesis, $F(t)-\frac c{f(t)^3}$ is smooth up to $t=1$. This is only possible
if all the log-terms in~\eqref{tXA} vanish, i.e.~$2A_{m+2}-A_{m+1}=0$ for
all $m\ge1$, or $A_m=2^{2-m}A_2$ for all $m\ge2$. Feeding this into~\eqref{iv5} gives,
after summing a geometric series,
\begin{equation}
 \frac1{c_k} = (k+1) \Big[ 1+ \frac{A_1}{k+1} + \frac{2A_2}{(k+1)(2k+1)}\Big]
  + O(k^{-\infty})   \label{iv7} \end{equation}
(where ``$O(k^{-\infty})$'' means ``$O(k^{-N})$ $\forall N>0$''). Hence
\begin{align*}
F(t) &= \sum_{k=0}^\infty \frac{2k+1}{c_k}t^k   \\
&= \frac4{(1-t)^3} + \frac{2A_1-3}{(1-t)^2} + \frac{2A_2-A_1}{1-t} +H(t) \\
&= \frac4{L^3}+\frac{2A_1+3}{L^2}+\frac{A_1+2A_2+1}L +H(e^{-L}),
\end{align*}
where $H$ denotes a function (possibly a different one at each occurrence)
smooth up to $t=1$.
On~the other hand, from $f(1)=0$, $f'(1)=-1$ we have $f(t)=L+O(L^2)$; using the
hypothesis $F-\frac c{f^3}\in C^\infty\nuj$ again, we~see that necessarily $c=4$
and
$$ f=\Big(\frac c{F-H}\Big)^{1/3}
= L-\frac{2A_1+3}{12}L^2+\frac{4A_1^2+6A_1+3-12A_2}{72}L^3 + O(L^4),  $$
which implies by a laborious but routine calculation,
$$ \phi = t f' (f f'+t f f''-t f^{\prime2})
 = 1 - \frac{2A_1}3 L + \frac{4A_1^2+A_1-6A_2}{12}L^2 +a_3 L^3 + O(L^4)  $$
(we~will not need the value of~$a_3$). Thus $a_1=-\frac{2A_1}3$, while by \eqref{iv6} $a_1=-A_1$.
Hence necessarily $a_1=A_1=0$, and only $A_2$ (or~$a_2$) remains as a free parameter.
We~have thus obtained a one-parameter family of germs of $f$ at $t=1$ that can
satisfy the condition~\eqref{tUC}.

We~finish the proof by showing that the one-parameter family of germs of $\phi$
given by~\eqref{tTZ} is the one that leads to the coefficients $A_1=0$, $A_m=2^{2-m}A_2$
for $m\ge2$. Namely, by~\eqref{tTZ}, the coefficients $a_j$ for $\phi_v$ are given~by
$$ a_j = \frac{(1+\sqrt v)(1-\sqrt v)^j-(1-\sqrt v)(1+\sqrt v)^j}{4^jj!2\sqrt v}. $$
Hence
\begin{align*}
c_k &\approx \sum_{j=0}^\infty \frac{j!a_j}{(k+1)^{j+1}}  \\
&= \sum_{j=0}^\infty \frac{(1+\sqrt v)(1-\sqrt v)^j-(1-\sqrt v)(1+\sqrt v)^j}
  {4^j2\sqrt v (k+1)^{j+1}}  \\
&= \frac1{2(k+1)\sqrt v}\Big[
 \frac{1+\sqrt v}{1-\frac{1-\sqrt v}{4(k+1)}}
  - \frac{1-\sqrt v}{1-\frac{1+\sqrt v}{4(k+1)}} \Big]    \\
&= \frac{16k+8}{16k^2+24k+9-v} ,
\end{align*}
or
$$ \frac1{c_k} = \frac{2k^2+3k+\frac{9-v}8}{2k+1} + O(k^{-\infty}) .  $$
However, this is the same thing as~\eqref{iv7} with $A_1=0$ and $A_2=\frac{1-v}{16}$.
This completes the proof.   \end{proof}

\begin{proof}[Proof of Corollary~\ref{pB}] If~$f$ is real-analytic near~1,
then so is $\phi=W[f]$ and clearly, by~\eqref{tTZ}, also~$\phi_v$,
hence it follows from \eqref{tTH} that so~is~$h$; since $h^{(k)}(1)=0$
$\forall k$, we must have $h\equiv0$. Thus we have for some $c\neq0$
\begin{equation}
 \frac c{f(t)^3} = \sum_{k=0}^\infty \frac{(2k+1)\;t^k}
 {\int_0^1 t^k\phi(t)\,dt} ,   \label{clubsuit}
\end{equation}
where
\begin{equation}
 \phi=tf'(ff'+tff''-tf^{\prime2})=\phi_v.   \label{spadesuit}
\end{equation}
We~show this leads to a contradiction.

For $v<0$, say $v=-s^2$ with $s>0$, we~have $\phi_v(t)=e^{L/4} (\cos\frac{Ls}4
-\frac1s\sin\frac{Ls}4)$ where $L:=\log\frac1t>0$; this changes sign as $t\searrow0$,
contradicting the fact that $\phi(|z|^2)\,d\rho(z)= u^3\wedge^2(\frac i2\partial\dbar\log\frac1u)$
should be a nonnegative volume element on~$\MM$.
Thus $v\ge0$. In~that case, the~integral in~\eqref{clubsuit} is finite if and only~if
$k-\tfrac14-\tfrac14|\sqrt v|>-1$. Write
$$ \frac{|\sqrt v|-3}4 = m-1+\delta, \qquad m\in\{0,1,2,\dots\},\quad 0\le\delta<1. $$
The~integral is then finite precisely for $k\ge m$, and equals, by~a~small computation,
$$ \int_0^1 t^k \phi_v(t) \,dt = \frac{2k+1}{(2k+2m+2\delta+1)(k+1-m-\delta)} .  $$
For the right-hand side of \eqref{clubsuit} we thus get, again by a~small computation,
$$ \sum_{k=0}^\infty \frac{(2k+1)\;t^k} {\int_0^1 t^k\phi(t)\,dt} =
\frac{t^m(1+3t+4m(1-t)-\delta(4m+2\delta-1)(1-t)^2)}{(1-t)^3} .   $$
Taking reciprocals, we should thus have
$$ f(t)^3=c\frac{(1-t)^3}{t^m(1+3t+4m(1-t)-\delta(4m+2\delta-1)(1-t)^2)}, $$
or
$$ f(t)=c^{1/3}\frac{1-t}{\root3\of{t^m(1+3t+4m(1-t)-\delta(4m+2\delta-1)(1-t)^2)}}.$$
The~condition $f'(1)=-1$ implies that $c=4$, so finally
\begin{equation}
 f(t)=\frac{2^{2/3}t^{-m/3}(1-t)}{\root3\of{1+3t+4m(1-t)-\delta(4m+2\delta-1)(1-t)^2}}.
 \label{diamondsuit}   \end{equation}
This should now satisfy $tf'(ff'+tff''-tf^{\prime2})=\phi_v$.
From \eqref{diamondsuit} we see that $f(t)=t^{-m/3}e^{h(t)}$ with $h$ holomorphic at
the origin. This yields
$$ tf'(ff'+tff''-tf^{\prime2}) = -\tfrac13 e^{3h(t)} t^{-m}
 (m-3th'(t)) (h'(t)+t h''(t)).  $$
Thus $t^m\phi_v(t)$ should be holomorphic at $t=0$. This is clearly not the case when
$v=0$ (since $\phi_0$ contains $\log t$); and it also cannot be the case when
$0<|\sqrt v|\neq1$, since then $\phi_v$ is a linear combination of two powers of~$t$
(with nonzero coefficients) whose exponents sum to $-\frac12$. The~only case left is
$v=1$, corresponding to $m=0$, $\delta=\frac12$; however, then \eqref{diamondsuit} becomes
$$ f(t) = \frac{2^{2/3}(1-t)} {\root3\of{1+3t}} ,   $$
and
$$ tf'(ff'+tff''-tf^{\prime2}) = \frac{16t(1+t)(1+2t+5t^2)}{(1+3t)^4} ,  $$
which clearly does not equal $\phi_1(t)\equiv1$. This proves the corollary.
\end{proof}

\section{The general case}\label{sec4}
We~recall the following refinement of Lerch's formula~\eqref{tUL}, proved in~\cite{EbD}.

\begin{lemma}\label{pF} {\rm(Lemma~4 in~\cite{EbD})} The~series
\begin{equation}
  \sum_{k=1}^\infty \frac{t^k}{k^s} \Big(\log\frac1k\Big)^n
 = \Big(\frac d{ds}\Big)^n t\Phi(t,s,1) , \qquad n=0,1,2,\dots,
\label{tVA}  \end{equation}
equals
\begin{align}
&\sum_{j=0}^n \binom nj (-1)^{n-j} \Gamma^{(n-j)}(1-s) L^{s-1}(\log L)^j
  + \sum_{k=0}^\infty \zeta^{(n)}(s-k)\frac{(-1)^k}{k!}L^k,  \label{tVB} \\
&\hskip8em |L|<2\pi, \quad s\neq1,2,3,\dots, \quad L:=\log\frac1t. \nonumber
\end{align}
For $s=1,2,3,\dots$, the first sum on the right-hand side of the last
formula has to be replaced~by
\begin{equation}  \sum_{j=0}^n \binom nj c_{s,n-j} L^{s-1} (\log L)^j
 + \frac{(-1)^{s-1}}{(s-1)!}
   \Big[\gamma_n-\frac{(\log L)^{n+1}}{n+1}\Big] L^{s-1} ,  \label{tVC}
\end{equation}
while the term $k=s-1$ in the second sum on the right-hand side of
\eqref{tVB} has to be omitted. Here $c_{s,j}$ and $\gamma_j$ are certain
constants $($given explicitly below$)$.   \end{lemma}

Here $\gamma_j$ are the Stieltjes constants, i.e.
$$ \zeta(1+z) = \frac1z + \sum_{j=1}^\infty \frac{\gamma_j}{j!}z^j,
 \qquad z\in\CC,   $$
while $c_{m,j}$ are, similarly, the~coefficients of the Laurent expansion
of the Gamma function,
\begin{equation}
 \Gamma(1-m-z) = \frac{(-1)^m}{(m-1)!z} + \sum_{j=0}^\infty
  \frac{c_{m,j}}{j!} z^j, \qquad |z|<1, \; m=1,2,3,\dots.   \label{tVD}
\end{equation}
From the functional equation $\Gamma(z+1)=z\Gamma(z)$ for the Gamma function
we get the recurrence relations
\begin{equation}
c_{m+1,j}=-\frac{c_{m,j}+j c_{m+1,j-1}}m \quad\text{for }j\ge1,
 \qquad c_{m+1,0}=\frac{(-1)^m}{m!m}-\frac{c_{m,0}}m,  \label{tVE}
\end{equation}
with
\begin{equation}
 c_{1,j}=-\frac{c_{0,j+1}}{j+1}, \qquad c_{0,j}:=(-1)^j\Gamma^{(j)}(1). \label{tVF}
\end{equation}

For~later use, we~also note that the simple formula
$$ \int_0^\infty L^s e^{-kL} \,dL = \frac{\Gamma(s+1)}{k^{s+1}},
  \qquad \Re s>-1, \quad k=1,2,3,\dots,   $$
yields upon applying $(d/ds)^n$ to both sides ($n=0,1,2,\dots$)
\begin{equation}  \int_0^\infty L^s (\log L)^n e^{-kL} \,dL =
  \sum_{l=0}^n \binom nl \frac{\Gamma^{(n-l)}(s+1)}{k^{s+1}}
  \Big(\log\frac1k\Big)^l,   \label{tVI}  \end{equation}
by~the Leibniz rule.

\begin{proof}[Proof of Theorem~\ref{pC}] Assume that $u(z)=f(|z|^2)$ is a
smooth radial function on~$\MM^2$, with the asymptotic expansion \eqref{tTI}
satisfying~\eqref{tTJ}, which satisfies~\eqref{tUC}. Passing again from the
variable $t=|z|^2$ to $L=\log\frac1t$, \eqref{tTI}~and \eqref{tTJ} become
\begin{equation}
 f(t) \approx L \sum_{k=0}^\infty \sum_{j=0}^{M_k} a_{kj} L^k(\log\tfrac1L)^j,
 \qquad M_0=0, \quad a_{00}=1.   \label{tVG}  \end{equation}
We~will show that $M_k=0$ for all~$k$, so~that $f\in C^\infty\nuj$ as~claimed.

Assume, to~the contrary, that there is $N\ge1$ such that $M_0=M_1=\dots
=M_{N-1}=0$ but $M_N\ge1$ with $a_{NM_N}\neq0$.
By~\eqref{tVG} and the definition of~\eqref{tTI}, we~have
\begin{equation}
 u \approx L \Big[ 1 + p_N(L) + \sum_{j=1}^M a_j L^N(\log\tfrac1L)^j
  + O(L^{N+\delta}) \Big] ,   \label{tVH}  \end{equation}
with any $0<\delta<1$; here we started writing just $M$ and $a_j$ for
$M_N$ and~$a_{Nj}$, respectively, and $p_N$ stands for some polynomial
(not~necessarily the same one at each occurrence) of~degree $N$ without
constant term. Furthermore, \eqref{tVH}~can be differentiated termwise
any number of times.
Viewing, for the duration of this proof, $f$~temporarily as a function
of $L$ rather than of $t=e^{-L}$, the formula for $W[f]=\phi$ becomes
just $\phi=e^L f'(f^{\prime2}-ff'')$, and a routine computation gives
\begin{equation}
 \phi = 1+p_N(L)+\sum_{j=1}^M A_j L^N (\log\tfrac1L)^j+O(L^{N+\delta}),
 \label{tVK} \end{equation}
with
$$ A_M = (3-N)(N+1)a_M.  $$
For the moments $c_k=\int_0^1 t^k \phi(t)\,dt$ we~thus obtain, in~view of~\eqref{tVI},
$$ c_k\approx\frac1{k+1}+\frac{p_N(\frac1{k+1})}{k+1} + \sum_{j=1}^M
 \frac{A'_j}{(k+1)^{N+1}} (\log(k+1))^j + O\Big(\frac1{(k+1)^{N+1+\delta}}\Big), $$
with $A'_M=N!A_M$. Taking reciprocal and multiplying by~$2k+1$ gives
$$ \frac{2k+1}{c_k} \approx 2(k+1)^2 \Big[1+p_N(\tfrac1{k+1}) + \sum_{j=1}^M
 \frac{A''_j}{(k+1)^N} (\log(k+1))^j + O\Big(\frac1{(k+1)^{N+1+\delta}}\Big)\Big] $$
with $A''_M=-A'_M$. It~follows that
\begin{equation}
 F(t) = \sum_{k=0}^\infty \frac{2k+1}{c_k}t^k = 2\sum_{j=0}^N \beta_j \Phi(t,j-2,1)
 + 2\sum_{j=1}^M A''_j(-1)^j\Phi^{(j)}(t,N-2,1) + R(t), \label{tVJ}
\end{equation}
where $\Phi^{(j)}$ denotes the $j$-th derivative of $\Phi$ with respect to
the second argument, $\beta_j$~are some coefficients, and the remainder term
$R(t)$ is $O((1-t)^{N-3+\delta})$ for $N=1,2$, and belongs to $C^{N-3}\nuj$
if $N\ge3$.

If~$N=1$, \eqref{tVJ} becomes
$$ F=\frac4{L^3} \Big[1-2a_M L(\log\tfrac1L)^M+\dots\Big]  $$
where the dots stand for lower-order terms. On~the other hand,
$$ \frac1{f^3} = \frac1{L^3} \Big[1-3a_M L(\log\tfrac1L)^M+\dots\Big].  $$
Thus the condition $F-\frac c{f^3}\in C^\infty\nuj$ forces ($c=4$~and) $a_M=0$,
a~contradiction. Thus $N=1$ cannot occur.

If~$N=2$, \eqref{tVJ} becomes, after a bit more lengthy but completely
routine computation (using \eqref{tVA}--\eqref{tVF}) which we~omit,
$$ F=\frac4{L^3} \Big[1+\beta_1'L-3a_M L^2(\log\tfrac1L)^M
     - (3a_{M-1}-\tfrac{5M}2a_M)L^2(\log\tfrac1L)^{M-1}+\dots\Big] ,  $$
while, by~\eqref{tVH},
$$ \frac1{f^3}=\frac1{L^3}\Big[1+\beta_1''L-3a_M L^2(\log\tfrac1L)^M
     - 3a_{M-1}L^2(\log\tfrac1L)^{M-1}+\dots\Big]   $$
(the~dots again denote lower-order terms).
Hence $F-\frac c{f^3}\in C^\infty\nuj$ again forces ($c=4$~and) $\frac{5M}2a_M=0$,
contradicting the hypothesis that $M\ge1$ and $a_M\neq0$.
Thus $N=2$ cannot occur either.

If~$N\ge4$, \eqref{tVJ} becomes, using Lemma~\ref{pF},
\begin{align*}
F &= \frac4{L^3} \Big[1+\sum_{j=1}^{N-1}\beta_j'L^j
     +\sum_{j=3}^{N-1} \frac{A''_j}2 \frac{(-1)^j}{(j-3)!}L^j\log L  \\
&\hskip4em -\frac{N(N+1)(N-1)(N-2)(N-3)}{2(M+1)}(-1)^NL^N(\log\frac1L)^{M+1}
     +\dots\Big] ,
\end{align*}
while, by~\eqref{tVH},
$$ \frac1{f^3}=\frac1{L^3}\Big[1+\sum_{j=1}^{N-1}\beta_j''L^j
     -3a_M L^N(\log\tfrac1L)^M+\dots\Big].   $$
Thus $F-\frac c{f^3}\in C^\infty\nuj$ can only hold if ($c=4$, $A''_j=0$ for
all $j=1,\dots,N-1$, and) $N(N+1)(N-1)(N-2)(N-3)=0$, i.e.~$N\in\{1,2,3\}$,
a~contradiction again. Thus $N\ge4$ is likewise not possible.

We~are thus left with the case of $N=3$; note that in that case $A_M=0$ even
though $a_M\neq0$, so~a~bit more detailed analysis is needed.
Computing~$A_{M-1}$, \eqref{tVK}~gives
$$ \phi=1+p_2(L)+4Ma_M L^3(\log\tfrac1L)^{M-1}+\dots  $$
provided $M>1$. This gives, in~turn,
\begin{gather*}
c_k=\frac1{k+1}+\frac{p_2(\frac1{k+1})}{k+1}
 +\frac{24Ma_M}{(k+1)^4} (\log(k+1))^{M-1}+\dots,   \\
\frac{2k+1}{c_k}=2(k+1)^2\Big[1+p_2(\tfrac1{k+1})
 -\frac{24Ma_M}{(k+1)^3} (\log(k+1))^{M-1}+\dots\Big],
\end{gather*}
and
$$ F=\frac4{L^3} \Big[1+p_2(L) -12a_M L^3(\log\tfrac1L)^M+\dots\Big], $$
while by~\eqref{tVK},
$$ \frac1{f^3}=\frac1{L^3}\Big[1+p_2(L)-3a_M L^3(\log\tfrac1L)^M+\dots\Big]. $$
Hence the condition $F-\frac c{f^3}\in C^\infty\nuj$ forces ($c=4$~and) $a_M=0$,
a~contradiction as before.

This finally leaves us with the situation when $N=3$ and $M=1$, so~that
$$ f=L(1+\alpha_1L+\alpha_2L^2+aL^3\log\tfrac1L+\dots), \qquad a\neq0,  $$
yielding in turn
\begin{align*}
\phi &= 1+(4\alpha_1+1)L+(4\alpha_1+6\alpha_1^2+3\alpha_2+\tfrac12)L^2 \\
 &\hskip4em + (4a+2\alpha_1+6\alpha_1^2+4\alpha_1^3+3\alpha_2
              +10\alpha_1\alpha_2+\tfrac16)L^3+\dots,  \\
c_k &= \dfrac1{k+1}+\frac{4\alpha_1+1}{(k+1)^2}
       +\frac{2(4\alpha_1+6\alpha_1^2+3\alpha_2+\tfrac12)}{(k+1)^3} \\
 &\hskip4em + \frac{6(4a+2\alpha_1+6\alpha_1^2+4\alpha_1^3+3\alpha_2
              +10\alpha_1\alpha_2+\tfrac16)}{(k+1)^4}+\dots,   \\
\frac{2k+1}{c_k} &= 2(k+1)^2 \Big[1-\frac{4\alpha_1+\frac32}{k+1}
    + \frac{\frac12+2\alpha_1+4\alpha_1^2-6\alpha_2}{(k+1)^2} \\
 &\hskip4em +\frac{2\alpha_1^2+8\alpha_1^3-12\alpha_1\alpha_2-3\alpha_2-24a}
     {(k+1)^3} +\dots\Big] ,
\end{align*}
and
\begin{align*}
F &= \frac4{L^3} \Big[1+(\tfrac14-2\alpha_1)L+(2\alpha_1^2-\alpha_1-3\alpha_2)L^2 \\
&\hskip6em  +(12a-\alpha_1^2-4\alpha_1^3+\tfrac32\alpha_2+6\alpha_1\alpha_2)L^3\log L
     +\dots\Big] ,
\end{align*}
while
$$ \frac1{f^3}=\frac1{L^3} \Big[1-3\alpha_1L+(6\alpha_1-3\alpha_2)L^2
      +3aL^3\log L+\dots\Big] .   $$
From $F-\frac c{f^3}\in C^\infty\nuj$ we thus get in turn, comparing the coefficients,
$c=4$, $\alpha_1=-\frac14$ and $a=0$, contradicting once again the assumption
that $a$ is nonzero. This completes the proof.  \end{proof}

\section{Concluding remarks}\label{sec5}
\subsection{Completeness}\label{sec51}
An~obvious distinction of $\MM$ from domains in $\CC^n$~is, of~course,
the~presence of the (removable) singularity at $z=0$. We~pause to note that,
in~fact, there can exist no balanced metric on~$\MM$ (radial or~not, smooth
up to $|z|=1$ or~not) that would be complete at the origin.

\begin{theorem} Let $u$ be a solution to the equation~\eqref{tTB}.
Then the balanced metric $\partial\dbar\log\tfrac1u$ is not complete at $z=0$.
\end{theorem}

\begin{proof} Recall that in some local coordinate chart, the above mentioned metric is given
explicitly by the coefficients
$$ g_{j\ok}(z) = \frac{\partial^2}{\partial z_j\partial\oz_k} \log\frac1{u(z)} ;  $$
and the length of a differentiable curve $\psi:(0,1)\to\MM$ is then given~by
\begin{equation}
 \int_0^1 \sqrt{\sum_{j,k} g_{j\ok}(\psi(x)) \psi'_j(x)\overline{\psi'_k(x)}} \,dx.
 \label{ddagger}
\end{equation}
We~apply this to the special case when $\psi(x)=xz$ is the segment joining the origin to
some fixed point $z$ of~$\MM$. The~last sum then equals, by~an elementary computation,
\begin{equation}
 \frac{\partial^2}{\partial x\partial\ox} \log\frac1{u(xz)} .  \label{dagger}
\end{equation}
Now~by the balanced condition~\eqref{tTB}, $\log c+(n+1)\log\frac1{u(z)}=\log K(z,z)$ where
$K=K_{u^{n+1}\wedge^n(\frac i2\partial\dbar\log\frac1u)}$ is the reproducing kernel.
Thus, up~to the (immaterial) constant factor $n+1$, \eqref{dagger}~equals
$$ \frac{\partial^2}{\partial x\partial\ox} \log K(xz,xz) .  $$
In~terms of any orthonormal basis $\{e_j\}$ of the corresponding Bergman space,
the~last kernel is given by $\sum_j|e_j(xz)|^2$.
Note that since the origin is a removable singularity in~$\HH$ (cf.~\cite[Theorem~2.4]{EUjk}
even for the more general situation then our ordinary Kepler manifold), each $e_j$ actually
extends to a holomorphic function in some neighborhood of the origin in $\CC^{n+1}$.
Let $m\ge0$ be the largest integer with the property that all the functions
$x\mapsto e_j(xz)$, $j=0,1,2,\dots$ (holomorphic in one complex variable~$x$),
vanish to order $m$ at $x=0$. We~then obtain
$$ K(xz,xz) = |x|^{2m} e^{F(x,x)}   $$
with some $F(x,y)$ holomorphic in $x,\overline y$ near $(x,y)=(0,0)$. Consequently,
$$ \frac{\partial^2}{\partial x\partial\ox} \log K(xz,xz)
 = \frac{\partial^2}{\partial x\partial\ox} F(x,x)   $$
is~(nonnegative and) continuous in a neighborhood of $x=0$.
Hence so will be its square root, and thus the integral \eqref{ddagger} is finite.
This means that there is a curve of finite length joining $z$ to the origin,
proving therefore that the metric is not complete at the origin.   \end{proof}

The last argument in fact shows that no complete balanced metric can exist
on a normal complex analytic space with singular locus of codimension $\ge2$.
% We omit the details.

\subsection{Poincar\'e metrics}\label{sec52}
Recall that, quite generally, for a K\"ahler metric $g_{j\ok}=\partial_j\dbar_k\Phi$
given by a potential~$\Phi$, the~volume element is given, in~the local chart,
by $g=\det[g_{j\ok}]=e^{(n+1)\Phi}J[e^{-\Phi}]$, so~that the Ricci tensor
$\Ric_{j\ok}=\partial_j\dbar_k\log g$ satisfies
$$ \Ric_{j\ok} = \partial_j\dbar_k J[u] + (n+1) g_{j\ok} \qquad(u=e^{-\Phi}).  $$
The~metrics with $\partial\dbar J[u]\equiv0$ therefore satisfy $\Ric_{j\ok}=(n+1)
g_{j\ok}$, i.e.~are \emph{K\"ahler-Einstein} metrics with constant~$(n+1)$.
Those for which $u$ in addition vanishes precisely to the first order at the boundary
(i.e.~$u=0$, $\nabla u\neq0$ on~$\pOm$) are usually called \emph{Poincar\'e} metrics
on~$\Omega$, cf.~\cite[Chapter~11]{BFG}. We~have seen in course of the proof of
Proposition~\ref{pD} that for $u(z)=f(|z|^2)$ on~$\MM$, $J[u](z)=W[f](|z|^2)/|z_{n+1}|^2$,
with $W[f]$ given by~\eqref{tTL}. It~follows, in~particular, that functions $f$ which
are solutions~to
\begin{equation}
 (-1)^n t f^{\prime(n-1)}(ff'+tff''-tf^{\prime2})\equiv c, \quad f(1)=0,\;f'(1)\neq0,
 \label{tWD}
\end{equation}
give rise to Poincar\'e metrics on~$\MM$. Replacing $f$ by an appropriate multiple,
we~can assume that $f'(1)=-1$ and $c=1$. Note that, in~particular, for $n=2$ \eqref{tWD}
is precisely the equation~\eqref{tTH} for $v=1$ and $h\equiv0$ (thus Poincar\'e metrics
are among the ``good candidates'' for a balanced metric, in~the radial situation).
In~this subsection, we~want to discuss the differential equation \eqref{tWD} in more detail.

For~simplicity, we~treat again in detail only the simplest case of $n=2$, so~that the
equation reads
$$ W[f]\equiv tf'(ff'+tff''-tf^{\prime2})=1.  $$
Differentiating the expression
$$ \Psi(t):=-\frac t{f(t)^3}+\frac{t^2 f'(t)^2}{2f(t)^2} - \frac{t^3 f'(t)^3}{f(t)^3}, $$
we~see that
$$ \Psi'(t) = \frac{f(t)-3t f'(t)}{f(t)^4} \; (W[f](t)-1).  $$
Thus $W[f]\equiv1$ implies $\Psi\equiv c$ for some constant~$c$.

The~equation $x^3+\frac12 x^2=a$ has for each $a\ge0$ a unique nonnegative root $x=:\rho(a)\ge0$.
From $\Psi\equiv c$ we thus obtain
\begin{equation}
 f'(t) = -\frac{f(t)}t \rho\Big(c+\frac t{f(t)^3}\Big)   \label{hv}
\end{equation}
as long as
\begin{equation}
 c+\frac t{f^3}\ge0. \label{hv2}
\end{equation}
The condition $f(1)=0$ means that this will be fulfilled as $t\nearrow1$,
and by standard existence theorems (Pe\'ano --- note that since $\rho(a)\approx a^{1/3}$
when $a\to+\infty$, the function $(t,y)\mapsto-\frac yt \rho(c+t y^{-3})$ is continuous
near $(t,y)=(1,0)$) applied to \eqref{hv} yield a unique solution $f(t)$ for $t\in(t_0,1)$
with some $0<t_0<1$. As~long as $\rho>0$, we~will have $f'<0$, so $f$ will be decreasing
and, hence, positive. When $c\ge0$, it~is possible to continue in this way down to $t=0$
(\eqref{hv2}~will still be fulfilled), and for $t\to0+$ we will have $f'(t)\approx -\rho(c) f(t)/t$,
or $f(t)\approx A t^{-\rho(c)}$. When $c<0$, the~solution reaches for some $t_0\in(0,1)$
the situation when $c+\frac t{f(t)^3}=0$, whence $f'(t_0)=0<f(t_0)$; from~$W[f](t)=0$ we then
obtain $f''(t_0^+)=-\infty$, so $f$ develops a singularity at $t=t_0$ and the solution
terminates there (cf.~also Remark~\ref{R11} below).

Altogether, we~thus arrive at a family~$f_c(t)$, $c\ge0$, of~Poincar\'e metrics on~$\MM$,
smooth up to the outer boundary $|z|=1$.

\begin{remark} The constant $c$ appears in the Taylor expansion of $f$ at 1, but only
in the fourth derivative: namely, taking the Taylor expansion of $\Psi(t)-c=0$ at $t=1$
yields in turn
$$ f(1)=0, \quad f'(1)=-1, \quad f''(1)=\tfrac12, \quad f'''(1)=-\tfrac34,
 \quad f''''(1)=\tfrac{15+16c}8 .  $$
The~condition on the solution $f$ to reach as far as $t=0$ thus is $f''''(1)\ge\tfrac{15}8$.
\qed   \end{remark}

\begin{remark} The~only explicit solution we know is $f(t)=2-2\sqrt t$, corresponding to $c=0$.
For~general~$n$, an~explicit solution to \eqref{tWD} is $g(t)=\frac n{n-1}(1-t^{(n-1)/n})$.
\qed   \end{remark}

\begin{remark} For~$c\ge0$, from $f_c\sim t^{-\rc}$ as $t\searrow0$ we see that
the volume density $g=u^{-3}J[u]\sim t^{3\rc}$ is nonvanishing at the origin only
for $c=0$ (and then we know explicitly that $f(t)=2-2\sqrt t$, by~the preceding remark).
For~$c>0$, observing that $\rho$ is $C^\infty$ in a neighborhood of~$c$ and writing
$f(t)=t^{-\rc}h(t)$ with $h(0)\neq0$, we~get $\frac{tf'}f=\frac{th'}h-\rc$ so~that
\eqref{hv} becomes
$$ \frac{th'}h=\rc-\rho\Big(c+\frac{t^{3\rc+1}}{h^3}\Big)\approx
 -\rho'(c)\frac{t^{3\rc+1}}{h^3} = -\frac{t^{3\rc+1}}{h^3\rc(3\rc+1)} , $$
since $\rho'=1/(3\rho^2+\rho)$. Solving for $h$ gives
$$ h\approx h(0)-\frac{t^{3\rc+1}}{h(0)^2\rc(3\rc+1)^2}.  $$
Continuing in this fashion ultimately yields
$$ f(t)=t^{-\rc} Q(t^{3\rc+1})  $$
for some function $Q\in C^\infty[0,1]$. This gives a complete description of
the Poincar\'e metric corresponding to $c>0$ in the neighborhood of the
singularity at the origin.

Note that, by~a~similar argument as in the preceding subsection, for any $c\ge0$
the corresponding Poincar\'e metric is incomplete at the origin $z=0$.
In~fact, the integrand in~\eqref{ddagger} behaves for $x\searrow0$ as $x^{3\rc}$
for $c>0$, and as $x^{-1/2}$ for $c=0$, hence the integral is always finite. \qed
\end{remark}

\begin{remark} \label{R11}
For~$c<0$, we~have seen that as $t$ decreases from~1, we~reach at
some $t=t_0$ the situation when $t_0>0$, $f(t_0)>0$ but $c+\frac{t_0}{f(t_0)^3}=0$,
whence $f'(t_0)=0$ and $f''(t_0^+)=-\infty$.
A~similar analysis as in the preceding remark, starting from the observation that
$$ \rho(x)\approx \sqrt{2x}-2x+5\sqrt2 x^{3/2}-32x^2+\dots $$
is~a smooth function of $\sqrt x$ at the origin, shows that
$$ f(t)=Q(\sqrt{t-t_0})   $$
for some $Q\in C^\infty[0,\sqrt{1-t_0}]$, with
$$ Q(0)=f(t_0), \;Q'(0)=Q''(0)=0 \quad\text{and}\quad Q'''(0)=\frac{4\sqrt2}{t_0\sqrt{f(t_0)}}.$$
In~particular, as~$t\searrow t_0$,
$$ f'(t)\approx -\frac{\sqrt{2(t-t_0)}}{t_0\sqrt{f(t_0)}}, \qquad
 f''(t)\approx -\frac1{t_0} \sqrt{\frac{t-t_0}{2f(t_0)}} .   $$
In~particular, the~solution $f$~cannot be continued in any way across $t=t_0$
--- it~would have to assume imaginary values for $t<t_0$.
(The~same~is, of~course, true for $c=0$ and $t_0=0$, when $f(t)=2-2\sqrt t$.) \qed
\end{remark}

\subsection{Erratum}\label{sec53}
We~conclude by giving a fix for a small overlook in the proof of Theorem~3 in~\cite{EbD}:
the~argument treating the case $N=1$ after (42) there exhibits a contradiction by
taking $j=M-1$, where $j\in\{1,\dots,M\}$, $M\ge1$. This is fine for $M\ge2$,
but makes no sense for $M=1$.

To~handle the overlooked case of $N=1$, $M=1$, we~make the computations
there in greater detail: namely, starting again with (as~before, the~dots always
denote lower order terms)
\begin{equation}
 f= L(1+b L\log L+aL+\dots), \qquad b\neq0,   \label{tWG}
\end{equation}
gives, in~turn, in~the notations of~\cite{EbD} ($C$~is the Euler constant)
\begin{align*}
w &\equiv u^2\partial\dbar\log\frac1u = 1+2bL\log L+(2a-b+1)L+\dots,  \\
c_k &= \frac1{k+1}-\frac{2b}{(k+1)^2}\log(k+1) + \frac{2a+b+1-2bC}{(k+1)^2}+\dots, \\
\frac1{c_k} &= (k+1) \Big[1-\frac{2b}{k+1}\log(k+1) + \frac{2Cb-2a-b-1}{(k+1)^2} + \dots\Big],  \\
F &= \frac1{L^2} \Big[1-2b L\log L - (2a+b)L + \dots \Big] ,
\end{align*}
and
$$ f^2F = 1 - bL + \dots .  $$
The~condition that $f^2F-c\pi$ is smooth up to $t=1$ and vanishes to second
order there thus implies that ($c=\frac1\pi$~and) $b=0$, contradicting the
hypothesis $b\neq0$ in~\eqref{tWG}. This completes the proof.

% \section*{Acknowledgement} The author thanks the anonymous referee for his
% comments on the original version of this manuscript.

% \section*{Funding} Research supported by GA \v CR grant no.~16-25995S and
% by~RVO funding for I\v{C} 67985840.

\end{document}